\documentclass[a4paper,12pt]{article}

\usepackage{amssymb,amsmath,amsthm}
\usepackage{graphics,graphicx,color}
\usepackage{srcltx}

\newtheorem{theorem}{Theorem}
\newtheorem{lemma}{Lemma}
\newtheorem{corollary}{Corollary}
{\rm}

\begin{document}

\title{A quasispecies continuous contact model in a subcritical regime
\thanks{The research was performed at the
Institute for Information Transmission Problems.
}
}

\author{Sergey Pirogov\thanks{
Institute for Information Transmission Problems, Moscow, Russia
(s.a.pirogov@bk.ru).} \and Elena Zhizhina\thanks{Institute for
Information Transmission Problems, Moscow, Russia (ejj@iitp.ru).} }

\date{\it
The paper is dedicated to the memory of Robert Adolfovich Minlos, outstanding mathematician and generous person.}

\maketitle

\begin{abstract}
We study a non-equilibrium dynamical model: a marked continuous contact model in $d$-dimensional
space, $d \ge 1$.
In contrast with the continuous contact model in a critical regime, see \cite{KKP}, \cite{KPZ}, the model under consideration is in the subcritical regime and it contains an additional spontaneous spatially homogeneous birth from an external source. We prove that this
system has an invariant measure. We prove also that the process starting from any initial distribution converges to this invariant measure.
\\

Keywords: continuous contact model; marked configurations;  correlation functions; statistical dynamics
\end{abstract}

\section{Introduction}

In this paper we study a marked continuous contact model in $d$-dimensional
space, $d \ge 1$, with a spontaneous birth rate. This model can be considered as a
special case of birth-and-death processes in the continuum,
\cite{KKP, KPZ, KS}, and it is inspired by the concept of quasispecies in population genetics, see e.g. \cite{ES, N}.
The phase space of such processes is the space $
\Gamma = \Gamma (R^d \times S)$ of locally finite marked
configurations in $R^d$ with marks $s \in S$ from a compact metric
space $S$. We describe here the stationary
regime and specify relations between solutions of the Cauchy
problem and this stationary regime in the subcritical case, i.e. when
the mortality prevails over reproduction and the deficiency is compensated by the immigration.

The analysis of the model is based on the concept of statistical dynamics. Instead of the construction of the stochastic dynamics as a Markov process on the configuration space, we use here the formal generator of the model for the derivation of the evolution equation for the correlation functions similar to the BBGKY hierarchy for Hamiltonian dynamics. In the general framework of \cite{FKK} we use the hierarchy of equations for time-dependent correlation functions to describe the Markov dynamics of our system, c.f. also \cite{KKP}, \cite{KPZ}.

With biological point of view, the stochastic system under study is
a model of an asexual reproduction under mutations
and selections, where an individual at the point $u \in R^d$ with
the genome $s \in S$ produces an offspring distributed in the
coordinate space and in the genome space with the rate $\alpha(u-v)
Q(s,s')$. The function $Q(s,s')$ is said to be the mutation kernel.
Since mortality exceeds reproduction, the model contains
a spontaneous additional birth that can be interpreted as an exogenous flow
generated by an external source, or an immigration.


\section{The model and the main results}


We consider a quasispecies contact model in a sub-critical regime with spontaneous birth on $M = R^d \times S$,
$d \ge 1$, $S$ is a compact metric space. A heuristic
description of the process is given by a formal generator:
\begin{equation}\label{generator}
(L F)(\gamma) = \sum_{ x \in \gamma}(F(\gamma \backslash x) - F(
\gamma)) +  \int\limits_M \Big(\sum_{y \in \gamma} \kappa a(x,y) + c(x) \Big) (F(\gamma \cup
x ) - F(\gamma)) dZ(x),
\end{equation}
where  $dZ = d\lambda d\nu $ is a product of the Lebesgue measure $
\lambda$ on $R^d$ and a finite Borel measure $\nu$ on $S$ with
$supp \ \nu =S$. Below, see  (\ref{59})-(\ref{korf}),  we will construct
the operator $\hat L^{\ast}$ describing the evolution of the correlation functions
(the BBGKY type hierarchy equations).

Here $b(x, \gamma) = \kappa  \sum_{y \in \gamma} a(x,y)$
are birth rates related to the contact model, $c(x)$ is a spontaneous birth rate (immigration), and
mortality rate equals 1.
We take $a(x,y)$ in the following form:
\begin{equation}\label{a}
a(x,y) = \alpha (\tau(x) - \tau(y)) \, Q(\sigma(x), \sigma(y)),
\end{equation}
$\tau$ and $\sigma$ are projections of $M$ on $R^d$ and $S$
respectively, $\alpha(u) \ge 0$ is a function on $R^d$ such that
\begin{equation}\label{2a}
\int\limits_{R^d} \alpha(u) du \ = \ 1, \quad   \int\limits_{R^d} |u|^2 \alpha(u) du \ < \ \infty,
\end{equation}
the covariance matrix $C$
\begin{equation}\label{2d}
C_{jk} \ = \ \int\limits_{R^d} u_j u_k \alpha(u) du \ - \ m_j m_k, \quad
m_j \ = \ \int\limits_{R^d} u_j \alpha(u) du,
\end{equation}
is non-degenerate and
\begin{equation}\label{2b}
 \hat \alpha (p) \ = \ \int\limits_{R^d} e^{i(p,u)} \alpha(u) du \in
L^1(R^d).
\end{equation}
It follows in particular that $|\hat \alpha (p)|<1$ for all $p \neq
0$.

We assume in what follows that $c(x)=c(\sigma(x)), \ \sigma(x) \in S $, i.e. we consider the spatially homogeneous spontaneous birth rates.
We suppose that the function $Q$ is continuous on
$S \times S$ (and so bounded) and strictly positive.
We consider the integral operator
\begin{equation}\label{Q}
(Qh)(s) \ = \ \int_S Q(s, s') h(s') d\nu(s'), \quad h \in C(S),
\end{equation}
in the Banach space of continuous function $C(S)$.
Then the Krein-Rutman theorem
\cite{KR} implies that there are a positive number $r>0$ and a
strictly positive continuous function $q(s) \in C(S)$, such that $Qq \
= \ rq$. The spectrum of $Q$, except $r$, which is a discrete spectrum
(accumulated to 0, if $S$ is not finite), is contained in the open disk $\{ z: \ |z|<r \}
\subset \mathbb{C}$. This "rest spectrum" is the spectrum of $Q$ on the subspace
"biorthogonal to $q$", i.e. on the subspace of the functions $h(s)$
such that
\begin{equation}\label{F1}
\int_S h(s) \ \tilde q(s) \  d\nu(s)\  =  \ 0.
\end{equation}
Here $\tilde q(s)$ is the strictly positive eigenfunction of the
adjoint operator $Q^{\star} (s,s') = Q(s', s)$.

We assume in what follows that $\kappa < \kappa_{cr}, \ \kappa_{cr}= r^{-1},$
and now including $\kappa_{cr}$ in $Q$ we suppose
that $r=1$, i.e. $Qq=q$. So the "renormalized" critical value of
$\kappa$ equals 1. We also normalize the function $q$ by the condition
\begin{equation}\label{norm}
\int_S q(s) d\nu(s) \ = \ 1.
\end{equation}

Note that the existence problem for Markov processes in $\Gamma$ for
general birth and death rates is an essentially open problem, but for the contact processes it was solved in \cite{KS}.
An alternative way of studying the evolution of the system is to
consider the corresponding statistical dynamics. The latter means
that instead of a time evolution of configurations we consider a
time evolution of initial states (distributions), i.e. solutions of
the corresponding forward Kolmogorov (Fokker-Planck) equation, see
(\cite{FKK, KKZ}) for details.

We should remind basic notations and constructions to derive time
evolution equations on correlation functions of the considered
model. Let ${\cal B}(M)$ be the family of all Borel sets in $M = R^d
\times S$, and ${\cal B}_b (M) \subset {\cal B}(M)$ denotes the
family of all bounded sets from ${\cal B}(M)$. The configuration
space $\Gamma (M)$ consists of all locally finite subsets of $M$:
\begin{equation}\label{F2}
\Gamma \ = \  \Gamma (M) \ =  \ \{ \gamma \subset M: \ |\gamma \cap
\Lambda| < \infty \; \mbox{ for all } \; \Lambda \in {\cal B}_b(M)
\}.
\end{equation}
Together with the configuration space $\Gamma(M)$ we define the
space of finite configurations
\begin{equation}\label{F3}
\Gamma_0 \  = \  \Gamma_0 (M) \ = \ \bigsqcup_{n \in N \cup \{0 \} }
\ \Gamma_0^{(n)},
\end{equation}
where $\Gamma_0^{(n)}$ is the space of $n$-point configurations
$
\Gamma_0^{(n)} \ = \ \{ \eta \subset M : \ |\eta| = n
\}.
$
The space $ \Gamma_0 (M) $ is equipped by the Lebesgue-Poisson measure $\exp (dZ) = 1 + dZ + \frac{dZ \otimes dZ}{2!}+ \ldots $.

We denote 
the set of cylinder functions on $\Gamma$ by ${\cal F}_{cyl}(\Gamma)$.
Each $F \in {\cal F}_{cyl}(\Gamma)$ is characterized by the following relation:
$F(\gamma) = F(\gamma \cap \Lambda)$ for some $\Lambda \in {\cal B}_b
(M)$.
The notation $B_{bs}(\Gamma_0)$ is used for
the set of bounded measurable functions with bounded support, i.e.  $G \in B_{bs}(\Gamma_0)$, if $G$ is a bounded measurable function on $\Gamma_0$, and there exists  $\Lambda \in {\cal B}_b(M)$ and $N \in \mathbb{N}$ such that
\begin{equation}\label{F4}
G|_{\Gamma_0 \backslash \ \bigsqcup_{n=0}^N \ \Gamma_\Lambda^{(n)}} = 0,
\end{equation}
where  $\Gamma_\Lambda^{(n)} \ = \ \{ \eta \subset \Lambda : \ |\eta| = n\}$ is the space of $n$-point configurations from $\Lambda$.

Next we define a mapping from $B_{bs}(\Gamma_0)$ into ${\cal
F}_{cyl}(\Gamma)$ as follows:
\begin{equation}\label{FF}
(K \ G)(\gamma) \ = \ \sum_{\eta \subset \gamma} G(\eta), \quad
\gamma \in \Gamma, \; \eta \in \Gamma_0,
\end{equation}
where the summation is taken over all finite subconfigurations $\eta
\in \Gamma_0$ of the infinite configuration $\gamma \in \Gamma$, see
i.g. \cite{KKP} for details.
Let us remark that this mapping is linear, positivity preserving and injective.
It is called K-transform.

In the same way as in \cite{KKP} we conclude that
the operator $\hat L = K^{-1}LK$ (the
image of $L$ under the K-transform) on functions   $G \in
B_{bs}(\Gamma_0)$ has the following form:
\begin{equation}\label{prop1}
\begin{array}{l}\displaystyle
(\hat L G)(\eta) \ = \  - |\eta | G(\eta)  \  + \    \int\limits_M \kappa
\sum_{x \in \eta} a(y,x) G((\eta \backslash x) \cup y) dZ(y)
\\ [5mm] \displaystyle
+ \int\limits_M \kappa \sum_{x \in \eta} a(y,x) G(\eta \cup y) d Z(y) + \int\limits_M c(y) G(\eta \cup y) d Z(y).
\end{array}
\end{equation}

\medskip

Denote by ${\cal M}^1_{fm}(\Gamma)$ the set of all probability
measures $\mu$ which have finite local moments of all orders, i.e.
\begin{equation}\label{F5}
\int\limits_{\Gamma} |\gamma_{\Lambda}|^n \ \mu (d \gamma) \ < \ \infty
\end{equation}
for  all $\Lambda \in {\cal B}_b(M)$ and $n \in N.$ If a measure
$\mu \in {\cal M}^1_{fm}(\Gamma)$ is locally absolutely continuous
with respect to the Poisson measure (associated with the measure
$dZ$), then there exists the corresponding system of the correlation
functions $k_{\mu}^{(n)}$ of the measure $\mu$, well known in
statistical physics, see e.g. \cite[Ch.4]{R}.
The set of correlation functions $k_{\mu}$ of the measure $\mu$ is defined as
\begin{equation}\label{F6}
\int\limits_{\Gamma} (K G)(\gamma) \mu (d \gamma) = \langle G, k_\mu \rangle,
\end{equation}
where $\langle , \rangle$ is the canonical duality between the functions and densities of measures on the space $\Gamma_0(M)$.

Let $\{ \mu_t \}_{t \ge 0} \subset {\cal M}_{fm}^1
(\Gamma)$ be the evolution of states described by the forward Kolmogorov equation with
the adjoint operator $L^{\ast}$. Then the generator of the evolution of the
corresponding system of correlation functions is defined as
\begin{equation}\label{F7}
\langle \hat L G, k \rangle \ = \  \langle G, \hat L^{\ast} k
\rangle, \quad  G \in B_{bs}(\Gamma_0),
\end{equation}
where the operator $\hat L = K^{-1} L K$ is defined by (\ref{prop1}).

The equations for the correlation functions have the following recurrent form:
\begin{equation}\label{59}
\frac{\partial k^{(n)}}{\partial t} \ = \ \hat L_n^{\ast} k^{(n)} \
+ \ f^{(n)}, \quad n\ge 1; \qquad k^{(0)} \equiv 1.
\end{equation}
Here $f^{(n)}$ are functions on $M^n$, $ f^{(1)}(x) = c(x) = c(\sigma(x)),$ and $f^{(n)}$ are defined for $n \ge 2$ as
\begin{equation}\label{f}
f^{(n)}(x_1, \ldots, x_n) \ = \ \sum_{i=1}^n  k^{(n-1)}(x_1,
\ldots,\check{x_i}, \ldots, x_n) \big(\sum_{j\neq i}^n \kappa a(x_i, x_j) + c(x_i) \big).
\end{equation}
These equations are the analogue of BBGKY equations for the considered system.

The operator $\hat L^{\ast}_n, \; n \ge 1, $ is
defined on the space $X_n = C \left( S^n, \ L^{\infty} ((R^d)^n) \right)$ as:
\begin{equation}\label{korf}
\begin{array}{l} \displaystyle
\hat L^{\ast}_n k^{(n)}(x_1, \ldots, x_n) \ = \ - n k^{(n)}(x_1,
\ldots, x_n)  \\ [3mm] \displaystyle
+ \ \sum_{i=1}^n \int\limits_M \kappa \, a(x_i, y) k^{(n)}(x_1, \ldots, x_{i-1}, y,
x_{i+1}, \ldots, x_n) d Z(y).
\end{array}
\end{equation}
We take any initial (for $t=0$) data
\begin{equation}\label{60}
k^{(n)}(0; x_1, \ldots, x_n) \in X_n.
\end{equation}

Invariant measures with finite moments of the contact process (if they exist) are described
in terms of the correlation functions $k^{(n)}$ on $M^n$ as a positive
solutions of the following system:
\begin{equation}\label{Last}
\hat L^{\ast}_n k^{(n)} + f^{(n)}=0, \quad n \ge 1, \quad
k^{(0)}\equiv 1,
\end{equation}
where $\hat L_n^{\ast}, \, f^{(n)}$ are defined in (\ref{f}) -
(\ref{korf}).

In this paper we prove the existence of the solution $k^{(n)} \in
X_n, \, n \ge 1$ of the system (\ref{Last}), such that $k^{(n)}$
have a specified asymptotics when $|\tau(x_i)-\tau(x_j)|\to\infty$
for all $i \neq j$. We also prove a strong convergence of the
solutions of the Cauchy problem (\ref{59}) - (\ref{60}) to the
solution of the system (\ref{Last}) of the stationary (time-independent)
equations.

\begin{theorem}\label{T1}
 1. Let the birth kernel $a(x,y)$ of the
contact model meet conditions (\ref{a})-(\ref{Q}), and $\kappa < \kappa_{cr}
= r^{-1}$.

Then for any positive continuous immigration rate $c(\sigma(x))$ there exists a
probability measure $\mu_c$ such that its system of
correlation functions $\{k_c^{(n)} \} $ is translation
invariant, solves (\ref{Last}), satisfies the following condition
\begin{equation}\label{Th1}
| k_c^{(n)} (x_1, \ldots, x_n) \ - \  \prod_{i=1}^{n}
k_c^{(1)}(x_i)| \ \to \ 0,
\end{equation}
when  $|\tau(x_i) - \tau(x_j)| \to \infty$ for all $i \neq j$, and
satisfies the following estimate
\begin{equation}\label{estimate}
k_c^{(n)} (x_1, \ldots, x_n) \  \le \  D \ H^n n!
\prod_{i=1}^n q(\sigma(x_i)) \quad \mbox{for any } \ x_1, \ldots,
x_n,
\end{equation}
for some positive constants $H=H(\kappa, Q, \alpha, c), \ D=D(\kappa, Q, \alpha, c)$. Here
$q(s)$ is the normalized eigenfunction of $Q$,
\begin{equation}\label{k1c}
k_c^{(1)}(x) \ = \ k_c^{(1)}(\sigma(x)) \ = \ \big( 1 - \kappa \, Q \big)^{-1} c(\sigma(x)).
\end{equation}

\medskip

2. For any $n \ge 1$ the solution $k^{(n)}(t)$ of the Cauchy
problem (\ref{59}) - (\ref{60}) converges to the solution
$k_c^{(n)}$ (\ref{Th1}) of the system (\ref{Last}) of
stationary (time-independent) equations as  $t \to \infty$:
\begin{equation}\label{Th1-2}
\| k^{(n)}(t) \ - \ k_c^{(n)} \|_{X_n} \ \to \ 0,
\end{equation}
where $X_n = C \left( S^n, \ L^{\infty}((R^d)^n) \right)$.
\end{theorem}

\section{The proof of Theorem \ref{T1}. Stationary problem. }

In this section we prove the first part of Theorem 1 using the
induction in $n$. For $n=1$ in (\ref{Last}) we have
\begin{equation}\label{8}
c(x)-k^{(1)}(x) + \int_M \kappa \, a(x,y) k^{(1)}(y) d Z(y) = 0, \quad \mbox{where } \; c(x) = c(\sigma(x)).
\end{equation}
Since we are constructing a translation invariant field we will look for
$k^{(1)}(x)$ in the form
\begin{equation}\label{F8}
k^{(1)}(x) \ = \ k^{(1)} (\sigma(x)).
\end{equation}
Then (\ref{8}) can be rewritten as
\begin{equation}\label{10}
c(s)- k^{(1)} (s) + \int_S \kappa \, Q(s,s') k^{(1)}(s') d\nu(s') \ = \ 0, \quad s = \sigma(x),
\end{equation}
which means that
\begin{equation}\label{k1-bis}
k^{(1)}(x) \ = \ k^{(1)}(\sigma(x)) = \big( 1 - \kappa \, Q \big)^{-1} c(\sigma(x)).
\end{equation}
\\

As a warm-up let us solve the equation (\ref{Last}) for the special
case $n=2, \; S=\{0\}, \; Q(0,0)=1, \; c(x) \equiv c$. This means that
$M=R^d$ and no marks. Then $k^{(1)}(x) = \frac{c}{1-\kappa}$ and the equation for $k^{(2)}(x)$ is written as
\begin{equation}\label{13}
\hat L^{\ast}_2 k^{(2)} + f^{(2)}=0,
\end{equation}
with
\begin{equation}\label{14}
f^{(2)}(x_1, x_2) \ = \ \kappa \, \varrho \, ( \alpha(x_1 - x_2) + \alpha(x_2 - x_1))\ + \ 2 \varrho  c,
\end{equation}
where $\varrho = \frac{c}{1-\kappa}$.
Thus $\hat L^{\ast}_2 \ = \  L^{(1)} + L^{(2)}$, where
\begin{equation}\label{15}
L^{(1)} k^{(2)}(x_1, x_2) \ = \ \int_{R^d} \kappa \, \alpha(x_1 - y)
k^{(2)}(y, x_2) dy - k^{(2)}(x_1, x_2),
\end{equation}
and analogously
\begin{equation}\label{16}
L^{(2)} k^{(2)}(x_1, x_2) \ = \ \int_{R^d} \kappa \, \alpha(x_2 - y)
k^{(2)}(x_1, y) dy - k^{(2)}(x_1, x_2).
\end{equation}
Using the translation invariance we have:
\begin{equation}\label{F9}
k^{(2)}(x_1, x_2) \ = \ k^{(2)}(x_1 - x_2).
\end{equation}
After the Fourier transform we can rewrite (\ref{13}) - (\ref{16})
as
\begin{equation}\label{18}
\big( \kappa \, \hat \alpha(p) + \kappa \, \hat \alpha (-p) -2 \big) \ \hat k(p)  \ = \ - \kappa \, \varrho \
\big(\hat \alpha(p) + \hat \alpha (-p) \big) - 2 c \varrho \, \delta (p).
\end{equation}
Therefore,
\begin{equation}\label{19}
\hat k(p)  \ = \ \frac{  \varrho \, \kappa (\hat \alpha(p) + \hat \alpha (-p))}
{2 -  \kappa (\hat \alpha(p) + \hat \alpha (-p))} + \hat A \delta(p),\; \mbox{ where } \; \hat A = \varrho^2.
\end{equation}

Now let us turn to the general case. If for any $n>1$ we succeeded to
solve the equation (\ref{Last}) and express $k^{(n)}$ through
$f^{(n)}$, then knowing the expression of $f^{(n)}$ through
$k^{(n-1)}$ via (\ref{f}), we would get the solution to the full system
(\ref{Last}). So we have to invert the operator $\hat L_n^{\ast}$.

Remind that
\begin{equation}\label{20}
\hat L_n^{\ast} \ = \ \sum_{i=1}^n L^i,
\end{equation}
where
\begin{equation}\label{21}
L^{i} k^{(n)}(x_1, \ldots, x_n) \ =
\end{equation}
\begin{equation}\label{F10}
\int_M \kappa\,  a(x_i, y) k^{(n)}(x_1, \ldots, x_{i-1}, y, x_{i+1}, \ldots,
x_n) d Z(y) - k^{(n)}(x_1, \ldots, x_n)
\end{equation}
are bounded operators in $X_n$.

\begin{lemma}\label{ Proposition 2.}
The operator $e^{t \hat L_n^{\ast}}$ is
monotone.
\end{lemma}

\begin{proof} The monotonicity of the operator $e^{t \hat
L_n^{\ast}}$ follows from (\ref{20}) - (\ref{21}):
\begin{equation}\label{F11}
e^{t \hat L_n^{\ast}} \ = \ \otimes_{i=1}^n  e^{t L^{i}}, \quad e^{t
L^{i}} \ = \ e^{-t} e^{t A_{i}},
\end{equation}
and the positivity of operators
\begin{equation}\label{F12}
A_i k^{(n)} \ = \ \int_M \kappa \, a(x_i, y) k^{(n)} (x_1, \ldots, x_{i-1}, y,
x_{i+1}, \ldots, x_n) d Z(y).
\end{equation}
\end{proof}

\begin{lemma}\label{L3}
If the function $f(x_1, \ldots, x_n)$ is bounded and satisfies the condition
\begin{equation}\label{condition-f}
f(x_1, \ldots, x_n) \to 0 \; \mbox{when} \; |\tau(x_i) - \tau(x_j)| \to \infty \; \mbox{for all} \; i \neq j \; \mbox{uniformly in} \; \sigma(x_1), \ldots,  \sigma(x_n),
\end{equation}
then the function
\begin{equation}\label{g}
g(x_1, \ldots, x_n) \ =  \ \Big(\int_0^{\infty} e^{t \hat L_n^{\ast}} f \ dt \Big) (x_1, \ldots, x_n) = \big( - \hat L_n^{\ast} \big)^{-1} f (x_1, \ldots, x_n)
\end{equation}
exists and satisfies the same condition.
\end{lemma}

\begin{proof}

First consider the restriction of $\hat L_n^{\ast}$ to the invariant
subspace $\Lambda$ consisting of the functions of the form
\begin{equation}\label{F13}
G_{\varphi, q}(x) = \varphi(\tau(x_1), \ldots, \tau(x_n)) \prod_{i=1}^n q(\sigma(x_i)),
\end{equation}
where $\varphi (w_1, \ldots, w_n) \in L^{\infty}((R^d)^n), \; Qq = q$.

The operator $\hat L_n^{\ast}$ acts on these functions as
\begin{equation}\label{24}
L_{n, max}  \ = \ \sum_{i=1}^n L^i_{max},
\end{equation}
where
\begin{equation}\label{25}
L^{i}_{max} \ G_{\varphi, q} (x) \  =
\end{equation}
\begin{equation}\label{F14}
\prod_{i=1}^n q(\sigma(x_i)) \left( \int_{R^d} \kappa \, \alpha(w_i-u)
\varphi(w_1, \ldots, w_{i-1}, u, w_{i+1}, \ldots, w_n) du - \varphi
(w_1, \ldots, w_n)\right)
\end{equation}
due to the equality $Qq=q$. Remind that
$\kappa_{cr}$ is "absorbed" in $Q$ and $\kappa<1$. Formula (\ref{25}) means that in
this case we have only spatial convolutions and no integration over
$S$.

Using that  $\hat L_n^{\ast} = \sum_{i=1}^n A_i - n$ and the spectral radii of operators $A_i, i=1, \ldots, n,$ equal to $\kappa<1$, we have
\begin{equation}\label{F15}
 \big( - \hat L_n^{\ast} \big)^{-1} \ = \ \frac{1}{n} \Big( 1 - \frac{1}{n} \sum_{i=1}^n A_i
 \Big)^{-1}.
\end{equation}
Notice that this formula again implies the monotonicity of the operator $\big( - \hat L_n^{\ast} \big)^{-1}$.
Denoting $A= \frac{1}{n} \sum_{i=1}^n A_i$  we get
\begin{equation}\label{F16}
\big( - \hat L_n^{\ast} \big)^{-1} \ = \  \frac{1}{n} \sum_{m=0}^{\infty} A^m.
\end{equation}
On the invariant subspace $\Lambda$ the operator $A$ is a convolution operator with the distribution
\begin{equation}\label{F17}
\beta(u) =  \frac{1}{n} \sum_{i=1}^n \kappa \, \alpha(u_i) \prod_{j \neq i} \delta(u_j), \; u=(u_1, \ldots, u_n), \; u_i \in R^d.
\end{equation}
Its integral equals to $\kappa<1$. Analogously the operator $A^m$ is a convolution with the distribution $\beta^{\ast m}$ whose integral is equal to $\kappa^m$. Consequently $\big( - L_{n, max} \big)^{-1}$ on $\Lambda$  is also the convolution with the distribution $\beta_{max}(u)$ whose integral is equal to $\frac{1}{n(1-\kappa)}$.
For $ G_{\varphi, q}(x)$ we have
\begin{equation}\label{F18}
\big( - L_{n, max} \big)^{-1} G_{\varphi, q}(x) \ = \ \frac{1}{n(1-\kappa)}  \prod_{i=1}^n q(\sigma(x_i)) \, {\mathbb E} \varphi (\tau(x)+ \xi),
\end{equation}
where $\tau(x) = (\tau(x_1), \ldots, \tau(x_n))$ and $\xi = (\xi_1, \ldots, \xi_n)$ is a random vector with the probability distribution $n(1-\kappa)\, \beta_{max}(u)$.

Thus if $f \in \Lambda$ satisfies condition \eqref{condition-f} of the Lemma, then the function ${\mathbb E} \varphi (\tau(x)+ \xi)$ satisfies the same condition by the Lebesgue dominated convergence theorem.

To prove the general case of the Lemma we notice that any function satisfying condition \eqref{condition-f} can be estimated in the absolute value by a function from $\Lambda$ satisfying the same condition. Then the conclusion of Lemma \ref{L3} follows from monotonicity of $\big( - \hat L_n^{\ast} \big)^{-1}$.
\end{proof}

From the proof of Lemma \ref{L3} we obtain the following estimate.

\begin{corollary}\label{corollary}
If a function $f(x_1, \ldots, x_n)$ is estimated in the absolute value by a function from the invariant subspace $\Lambda$ with
$|\varphi(w)| \le C$, then for the function $\big( - \hat L_n^{\ast} \big)^{-1} f$ the same estimate in the absolute value holds with constant $\frac{C}{n(1-\kappa)}$.
\end{corollary}


\medskip

Next we will construct a solution $k^{(n)} \  = \ \big( - \hat L_n^{\ast} \big)^{-1} f^{(n)}$ of the system (\ref{Last}) satisfying (\ref{Th1}) and meeting the estimate
\begin{equation}\label{est}
k^{(n)} (x_1, \ldots, x_n) \ \le \ K_n \prod_{i=1}^n q(\sigma(x_i))
\end{equation}
where $K_n = H^n n!, \ H$ is a constant.
We suppose by induction that
\begin{equation}\label{34}
k^{(n-1)} (x_1, \ldots, x_{n-1}) \ \le \ K_{n-1} \prod_{i=1}^{n-1}
q(\sigma(x_i)).
\end{equation}
Using (\ref{f}) we conclude that the function $f^{(n)}$ can be estimated as
\begin{equation}\label{F19}
f^{(n)}(x_1, \ldots, x_n) \ \le \  K_{n-1} C \, n^2 \prod_{i=1}^{n}
q(\sigma(x_i))
\end{equation}
with a constant $C>0$.

From Corollary \ref{corollary} it follows that  $k^{(n)} \  = \ \big( - \hat L_n^{\ast} \big)^{-1} f^{(n)}$ satisfies the same estimate \eqref{34} with $K_n = \frac{ C \, n}{1-\kappa} K_{n-1}$ and so
\begin{equation}\label{F20}
K_n = \frac{C^n \, n!}{(1-\kappa)^n}.
\end{equation}
Thus
\begin{equation}\label{est43}
k^{(n)} (x_1, \ldots, x_n) \ \le \  H^n n! \prod_{i=1}^n q(\sigma(x_i)) \quad \mbox{with } \; H= \frac{C}{1-\kappa}.
\end{equation}

To complete the first (stationary) part of Theorem \ref{T1} we need to verify \eqref{Th1}. We apply again the induction in $n$.
For $n=1$, it is trivial. If \eqref{Th1} is valid for $k^{(n-1)}$, then from \eqref{f} we obtain the following relation for $f^{(n)}$:
\begin{equation}\label{F21}
f^{(n)}(x_1, \ldots, x_n) \ - \ \sum_{i=1}^n c(x_i) \prod_{i \neq j} k^{(1)} (x_j) \ \to \ 0,
\end{equation}
when  $|\tau(x_i) - \tau(x_j)| \to \infty$ for all  $i \neq j$  uniformly in  $\sigma(x_1), \ldots,  \sigma(x_n)$.
Thus by Lemma \ref{L3} we get
\begin{equation}\label{forconv}
\big( - \hat L_n^{\ast} \big)^{-1} f^{(n)} (x_1, \ldots, x_n) \ - \ \big( - \hat L_n^{\ast} \big)^{-1} \sum_{i=1}^n c(x_i) \prod_{i \neq j} k^{(1)} (x_j) \ \to \ 0
\end{equation}
when  $|\tau(x_i) - \tau(x_j)| \to \infty$ uniformly in  $\sigma(x_1), \ldots,  \sigma(x_n)$.

We define
\begin{equation}\label{F22}
k_{as} (\sigma(x_1), \ldots, \sigma(x_n)) := \prod\limits_{i=1}^n k^{(1)} (x_i)
= \prod\limits_{i=1}^n \big( 1 - \kappa Q_i \big)^{-1} c(\sigma (x_i)),
\end{equation}
where $ k^{(1)} (x)$ was defined by \eqref{k1-bis} and  $Q_i$ is the operator $Q$ acting in the space of functions of $\sigma(x_i)$.
Using \eqref{korf}, \eqref{a} - \eqref{2a} and the fact that $k_{as}$ is a function of variables $\sigma (x_i)$ we conclude
\begin{equation}\label{F23}
(- \hat L_n^{\ast}) k_{as} (\sigma(x_1), \ldots, \sigma(x_n)) = n k_{as} (\sigma(x_1), \ldots, \sigma(x_n)) -  \sum_{i=1}^n \kappa Q_i  k_{as} (\sigma(x_1), \ldots, \sigma(x_n))
\end{equation}
\begin{equation}\label{F24}
=  \sum_{i=1}^n (1-\kappa Q_i)  \prod\limits_{j=1}^n ( 1 - \kappa Q_j )^{-1} c(\sigma (x_j)) = \sum_{i=1}^n c(\sigma (x_i)) \prod\limits_{j \neq i} k^{(1)} (x_j).
\end{equation}
Consequently,
\begin{equation}\label{F25}
(- \hat L_n^{\ast})^{-1} \sum_{i=1}^n c(\sigma (x_i)) \prod\limits_{j \neq i} k^{(1)} (x_j) =  k_{as} (\sigma(x_1), \ldots, \sigma(x_n)) = \prod\limits_{i=1}^n k^{(1)} (x_i)
\end{equation}
and by  (\ref{Last})
\begin{equation}\label{F26}
\big( - \hat L_n^{\ast} \big)^{-1} f^{(n)} (x_1, \ldots, x_n) = k^{(n)}(x_1, \ldots, x_n).
\end{equation}
Thus the convergence in \eqref{forconv} implies \eqref{Th1} for $ k^{(n)}$.



\medskip

Thus we proved the existence of solutions $\{ k_c^{(n)} \}$
of the system (\ref{Last}) corresponding to the stationary problem.
To verify that this system of correlation function is associated
with a measure $\mu_c$ on the configuration space, we will prove
in the next section that the measure $\mu_c$ can be constructed
as a limit of an evolution of measures $\mu (t)$
associated with the solutions of the Cauchy problem (\ref{59}) with
corresponding initial data (\ref{60}).

\section{The proof of Theorem \ref{T1}. The Cauchy problem.  }

In this section we find the solution of the Cauchy problem
(\ref{59}) - (\ref{60}) and prove the convergence (\ref{Th1-2}).
Using Duhamel formula we have
\begin{equation}\label{61}
k^{(n)}(t) \ = \  e^{t \hat L_n^{\ast}} k^{(n)}(0) \ + \  \int_0^t e^{(t-s) \hat
L_n^{\ast}} f^{(n)}(s) \ ds,
\end{equation}
where $f^{(n)}(s)$ is expressed through $k^{(n-1)}(s)$ via
(\ref{f}).

We have
\begin{equation}\label{64}
k^{(n)}(t) -  k_c^{(n)}  = e^{t \hat L_n^{\ast}} \big(k^{(n)}(0) - k_c^{(n)} \big)  +  \int_0^t
e^{(t-s) \hat L_n^{\ast}} \big( f^{(n)}(s) -  f_c^{(n)} \big) \ ds.
\end{equation}
Here $f_c^{(n)}$ is expressed in terms of $k_c^{(n-1)}$
by (\ref{f}), and the equation $ \hat L_n^{\ast}
k_c^{(n)} \ = \ - f_c^{(n)}$ implies that
\begin{equation}\label{F30}
\left( e^{t \hat L_n^{\ast}} - 1 \right) k_c^{(n)} \ = \ -
\int_0^t \frac{d}{ds} e^{(t-s) \hat L_n^{\ast}} k_c^{(n)} ds \
\ = \ - \int_0^t e^{(t-s) \hat L_n^{\ast}}  f_c^{(n)} \ ds
\end{equation}
We shall prove now that both terms in (\ref{64}) converge to 0 in
sup-norm of $X_n$.

First we remind that $ k^{(n)}(0) - k_c^{(n)}$ is the element of the space $X_n$, and thus it can be estimated in absolute value as
\begin{equation}\label{F31}
| k^{(n)}(0) -  k_c^{(n)} | \le A \prod_{i=1}^{n} q(\sigma(x_i)).
\end{equation}
On the right-hand side of this inequality there is the eigenfunction of the monotone operator $e^{t \hat L_n^{\ast}}$. Therefore,
\begin{equation}\label{F32}
\big|  e^{t \hat L_n^{\ast}} (k^{(n)}(0) -  k_c^{(n)}) \big|  \le e^{ - n(1-\kappa) t} A \prod_{i=1}^{n} q(\sigma(x_i))
\end{equation}
which tends to 0  as $t \to \infty$.


For estimation of the second term in \eqref{64} we use the induction. For $n=1$
we have $f^{(1)} = f_c^{(1)} = c(x)$ and hence
\begin{equation}\label{F35}
k^{(1)}(t) -  k_c^{(1)} = e^{t \hat L_n^{\ast}} (k^{(1)}(0) -  k_c^{(1)}).
\end{equation}
As we have seen this function tends to 0 in $X_n$-norm and
\begin{equation}\label{F36}
| k^{(1)}(t) -  k_c^{(1)} | \le \gamma_1 (t) q(\sigma(x)),
\end{equation}
where $\gamma_1(t) \to 0$ as $ t \to \infty$.

Suppose that
\begin{equation}\label{suppose}
| k^{(n-1)}(t) -  k_c^{(n-1)} | \le \gamma_{n-1} (t) \prod_{i=1}^{n-1} q(\sigma(x_i)),
\end{equation}
where $\gamma_{n-1}(t) \to 0$ as $ t \to \infty$. Then relation \eqref{f} implies that the similar estimate holds for the difference
 $ f^{(n)}(t) -  f_c^{(n)} $:
\begin{equation}\label{F37}
| f^{(n)}(t) -  f_c^{(n)} | \le \tilde \gamma_{n-1} (t) \prod_{i=1}^{n} q(\sigma(x_i)),
\end{equation}
with $\tilde\gamma_{n-1}(t) \to 0$.  Using the monotonicity of the operators $e^{s \hat L_n^{\ast}}$ we conclude that the second term in  \eqref{64} can be estimated from above in the absolute value by
\begin{equation}\label{F38}
\int_0^t e^{-n (1-\kappa) (t-s)}   \tilde \gamma_{n-1} (s) ds \, \prod_{i=1}^{n} q(\sigma(x_i)).
\end{equation}
We put
\begin{equation}\label{F39}
\gamma_n (t) = \int_0^t e^{-n (1-\kappa) (t-s) }  \tilde \gamma_{n-1} (s) ds.
\end{equation}
It is easy to see that $\gamma_n (t) \to 0$ as $t \to \infty$, hence estimate \eqref{suppose} is valid for $ k^{(n)}(t) $.

\medskip

Thus we proved the strong convergence (\ref{Th1-2}). Using results
from \cite{KS} we can conclude that the solution $\{ k^{(n)}(t) \}$ of the Cauchy problem (\ref{59}) is a system
of correlation functions corresponding to the evolution of states
$\{ \mu_t \}$. The construction of the measure $\mu_c$ from the family of correlations functions
\begin{equation}\label{F40}
k^{(n)}_c \ = \  \lim_{t\to\infty} k^{(n)} (t)
\end{equation}
is based on the Lenard positivity of this family,
see Remark 4.2 in \cite{KKP} about the reconstruction of probability measures by correlation functions.

\bigskip

{\bf Conclusions.}
1. In the subcritical regime with an immigration the contact model has a stationary measure depending on the immigration rate $c(s)$ and the birth rate $ \kappa a(x,y)$ and not depending on the initial state. \\
2. The stationary correlation functions $k_c^{(n)} (x_1, \ldots, x_n)$ of this model for large $n$ are estimated from above by $n!$ in contrast with the correlation functions of the critical contact model that can grow as $(n!)^2$, see \cite{KKP, KPZ}.

\end{document}